\documentclass[10pt]{amsart}
    \makeatletter 
        \newcommand{\addresseshere}{%
          \enddoc@text\let\enddoc@text\relax
        }
    \makeatother
\usepackage{graphicx}
\usepackage[style=alphabetic]{biblatex} 
\usepackage{tikz}
\usetikzlibrary{angles}
\usepackage{amssymb}
\usepackage{amsthm}
\usepackage[all]{xy}
\usepackage{microtype}
\usepackage{hyperref}
\usepackage{enumitem}
\usepackage[nameinlink]{cleveref}
\usepackage[dvipsnames]{xcolor}

\makeatletter
    
    \@addtoreset{equation}{section}

    \newcommand{\int@suppauth}{%
        \def\ifciteibid{\@secondoftwo}%
        \renewbibmacro*{cite:name}{}%
        \renewbibmacro*{cite:idem}{}%
        \renewbibmacro*{author}{}%
        \renewbibmacro*{editor+others}{}%
        \renewbibmacro*{translator+others}{}
    }
    \newcommand{\suppauth}{\AtNextCite{\int@suppauth}}
\makeatother

\hypersetup{
 colorlinks,
 linkcolor={teal},
 citecolor={teal},
 urlcolor={teal}
}

\calclayout

\DeclareFieldFormat
  [article,book,inbook,incollection,inproceedings,patent,thesis,unpublished]
  {title}{\emph{#1\isdot}}

\theoremstyle{plain}
	\newtheorem{theorem}{Theorem}
	\newtheorem{lemma}[theorem]{Lemma}

	\newtheorem{proposition}[theorem]{Proposition}

\theoremstyle{definition}

\def\Ch{\mathrm{Ch}}

\begin{document}
\title{At most 10 cylinders mutually touch:\\ a Ramsey-theoretic approach}
\author[T. Dillon]{Travis Dillon}
\address{Massachusetts Institute of Technology, Cambridge, MA, USA}
\email{dillont@mit.edu}
\author[J. Koizumi]{Junnosuke Koizumi}
\address{RIKEN iTHEMS, Wako, Saitama 351-0198, Japan}
\email{junnosuke.koizumi@riken.jp}
\author[S. Luo]{Sammy Luo}
\address{Massachusetts Institute of Technology, Cambridge, MA, USA}
\email{sammyluo@mit.edu}

\renewcommand{\leftmark}{T. DILLON, J. KOIZUMI, AND S. LUO}
\renewcommand{\rightmark}{\MakeUppercase{At most 10 cylinders mutually touch: a Ramsey-theoretic approach}}

\date{\today}
\thanks{}
\subjclass{52C17, 52A40, 05D10}

\begin{abstract}
Littlewood asked for the maximum number $N$ of congruent infinite cylinders that can be arranged in $\mathbb{R}^3$ so that every pair touches.
We improve upon the proof of the second author that $N \leq 18$ to show that $N \leq 10$.
Together with the lower bound established by Boz\'oki, Lee, and R\'onyai, this shows that $N \in \{7,8,9,10\}$.
Our method is based on linear algebra and Ramsey theory, and makes partial use of computer verification.
We also provide a completely computer-free proof that $N \leq 12$.

\end{abstract}

\null
\vspace{-\baselineskip}

\maketitle

\enlargethispage*{20pt}
\thispagestyle{empty}

\section{Introduction}

The following problem was originally posed by Littlewood \cite{Littlewood}: What is the maximum number $N$ of congruent infinite cylinders that can be arranged in $\mathbb{R}^3$ so that every pair of them touches?
Equivalently, this asks for the largest possible number of lines in $\mathbb{R}^3$ such that the distance between any two distinct lines is exactly $1$. We call $N$ the \emph{champagne glass number} because it is the maximum number of (idealized mathematical) champagne glasses that can simultaneously be joined in a toast.

Boz\'oki, Lee, and R\'onyai used computer-assisted techniques to prove that $N \geq 7$ \cite{BLR}. On the other hand, A.\@ Bezdek proved that $N \leq 24$ by geometric arguments \cite{Bezdek}. More recently, the second author improved this to $N \leq 18$ by introducing a new method that combines linear algebra with Ramsey theory \cite{Koizumi}.  

Our paper develops this method and establishes the sharper bound $N \leq 10$:

\begin{theorem}\label{main}
    There is no configuration of 11 lines in $\mathbb{R}^3$ in which the distance between every pair of lines is exactly 1.
\end{theorem}

Our argument begins with the same approach as in \cite{Koizumi}. For two directed lines $L,L'$ in $\mathbb{R}^3$ that are not coplanar, we define their \emph{chirality} $\varepsilon(L,L') \in \{\pm 1\}$ as illustrated below.

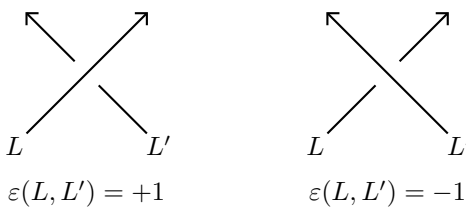
\begin{center}
\begin{tikzpicture}[thick, scale=0.8]
  \draw (-1,-1) -- (1,1);
  \draw(1,-1) -- (0.2, -0.2);
  \draw (-0.2, 0.2) -- (-1, 1);
  \draw (-1, 0.8) -- (-1,1) -- (-0.8, 1);
  \draw (1, 0.8) -- (1,1) -- (0.8, 1);
  \node at (-1.2,-1.2) {$L$};
  \node at (1.2,-1.2) {$L'$};
  \node at (0,-2) {$\varepsilon(L,L')=+1$};
  \begin{scope}[shift={(5,0)}]
  \draw (1,-1) -- (-1,1);
  \draw(-1,-1) -- (-0.2, -0.2);
  \draw (0.2, 0.2) -- (1, 1);
  \draw (-1, 0.8) -- (-1,1) -- (-0.8, 1);
  \draw (1, 0.8) -- (1,1) -- (0.8, 1);
  \node at (0,1.2) {};
  \node at (-1.2,-1.2) {$L$};
  \node at (1.2,-1.2) {$L'$};
  \node at (0,-2) {$\varepsilon(L,L')=-1$};
  \end{scope}
\end{tikzpicture}
\end{center}
Note that $\varepsilon(L,L') = \varepsilon(L',L)$.
If $\overline{L}$ is the line $L$ with its direction reversed, then $\varepsilon(\overline{L},L') = -\varepsilon(L,L')$.

We call a finite simple graph $G$ \emph{realizable} if there is a configuration of directed lines $\{L_v\}_{v \in V(G)}$ in $\mathbb{R}^3$ (called a \emph{realization} of $G$) such that  
\begin{itemize}
    \item no two lines are parallel,  
    \item the distance between any two lines is exactly $1$, and  
    \item $\varepsilon(L_v,L_w) = +1 \iff vw \in E(G)$.  
\end{itemize}
If $G$ is realizable, then its complement graph $\overline{G}$ is also realizable (for example, by taking the reflection of $\{L_v\}_{v \in V(G)}$ across a hyperplane).
The second author showed by a linear-algebraic argument that $K_5$ is not realizable \cite[Theorem 2]{Koizumi}. This, combined with a short Ramsey-theory argument, shows that $N \leq 18$.  

In this paper, we identify additional non-realizable graphs, which allows us to further improve the upper bound. In the following theorem, $C_n$ denotes the $n$-cycle, and the two graphs $H_6$ and $H_7$ are pictured in \Cref{fig:H6H7}.  

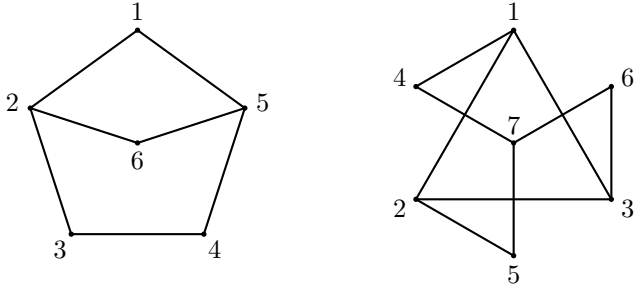
\begin{figure}[h]
    \centering
    \begin{tikzpicture}[scale=0.5]
      \foreach \i [evaluate=\i as \ang using {(\i-1)*360/5 + 90}] in {1,2,3,4,5}{
        \coordinate (P\i) at (\ang:3);
        \node at (\ang:3.5) {\i};
        \fill (P\i) circle (2pt);
      }
      \coordinate (P6) at (0,0);
      \node at (0,0) [below] {6};
      \fill (P6) circle (2pt);
      \foreach \i/\j in {1/2, 2/3, 3/4, 4/5, 5/1, 2/6, 5/6}{
        \draw[thick] (P\i) -- (P\j);
      }
      \begin{scope}[xshift=10cm]
      \foreach \i [evaluate=\i as \ang using {\i*360/3 - 30}] in {1,2,3}{
        \coordinate (P\i) at (\ang:3);
        \node at (\ang:3.5) {\i};
        \fill (P\i) circle (2pt);
      }
      \coordinate (P4) at (150:3);
      \node at (150:3.5) {4};
      \fill (P4) circle (2pt);
      \coordinate (P5) at (270:3);
      \node at (270:3.5) {5};
      \fill (P5) circle (2pt);
      \coordinate (P6) at (30:3);
      \node at (30:3.5) {6};
      \fill (P6) circle (2pt);
      \coordinate (Q) at (0,0);
      \node at (0,0) [above] {7};
      \fill (Q) circle (2pt);
      
      \foreach \i/\j in {1/2, 2/3, 3/1, 1/4, 2/5, 3/6}{
        \draw[thick] (P\i) -- (P\j);
      }
      \foreach \i in {4,5,6}{
        \draw[thick] (P\i) -- (Q);
      }
      \end{scope}
    \end{tikzpicture}
    \caption{$H_6$ (left) and $H_7$ (right)}
    \label{fig:H6H7}
\end{figure}

\begin{theorem}\label{forbidden_subgraphs}
    The graphs $K_5$, $K_{3,2}$, $K_7\setminus C_5$, $K_7\setminus H_6$, and $K_8\setminus H_7$ are not realizable.
\end{theorem}

We will also use the following Ramsey-theoretic result on closely related graphs.

\begin{theorem}\label{ramsey}
    In every red-blue coloring of the edges of $K_{10}$, either the red or blue subgraph contains an induced subgraph isomorphic to one of $K_4$, $K_{3,2}$, $K_6\setminus C_5$, $K_6\setminus H_6$, or $K_7\setminus H_7$.
\end{theorem}

\begin{figure}[h]
    \centering
    \begin{tikzpicture}[scale=0.5]
    \begin{scope}
      \foreach \i [evaluate=\i as \ang using {(\i-1)*720/5 + 90}] in {1,...,5}{
        \coordinate (P\i) at (\ang:3);
        \node at (\ang:3.5) {\i};
        \fill (P\i) circle (2pt);
      }
      \coordinate (P6) at (0,0);
      \node at (0,-0.7) {$6$} ;
      \fill (P6) circle (2pt);
      \foreach \i/\j in {1/3, 1/4, 2/4, 2/5, 3/5, 1/6, 2/6, 3/6, 4/6, 5/6}{
        \draw[thick] (P\i) -- (P\j);
      }
    \end{scope}
    \begin{scope}[xshift=10cm]
      \foreach \i [evaluate=\i as \ang using {(\i-1)*720/5 + 90}] in {1,...,5}{
        \coordinate (P\i) at (\ang:3);
        \node at (\ang:3.5) {\i};
        \fill (P\i) circle (2pt);
      }
      \coordinate (P6) at (0,0);
      \node at (0,-0.7) {$6$} ;
      \fill (P6) circle (2pt);
      \foreach \i/\j in {1/3, 1/4, 2/4, 2/5, 3/5, 1/6, 3/6, 4/6}{
        \draw[thick] (P\i) -- (P\j);
      }
    \end{scope}
    \begin{scope}[xshift=20cm]
      \coordinate (P1) at (90:3);
      \node at (90:3.5) {1};
      \fill (P1) circle (2pt);
      \coordinate (P2) at (210:3);
      \node at (210:3.5) {2};
      \fill (P2) circle (2pt);
      \coordinate (P3) at (330:3);
      \node at (330:3.5) {3};
      \fill (P3) circle (2pt);
      \coordinate (P4) at (270:3);
      \node at (270:3.5) {4};
      \fill (P4) circle (2pt);
      \coordinate (P5) at (30:3);
      \node at (30:3.5) {5};
      \fill (P5) circle (2pt);
      \coordinate (P6) at (150:3);
      \node at (150:3.5) {6};
      \fill (P6) circle (2pt);
      \coordinate (P7) at (0,0);
      \node at (0,-0.7) {$7$} ;
      \fill (P7) circle (2pt);
      \foreach \i/\j in {1/5, 1/6, 1/7, 2/4, 2/6, 2/7, 3/4, 3/5, 3/7, 4/5, 5/6, 6/4}{
        \draw[thick] (P\i) -- (P\j);
      }
    \end{scope}
    \end{tikzpicture}
    \label{fig:forbidden}
    \caption{$K_6\setminus C_5$ (left), $K_6\setminus H_6$ (middle), $K_7\setminus H_7$ (right)}
\end{figure}
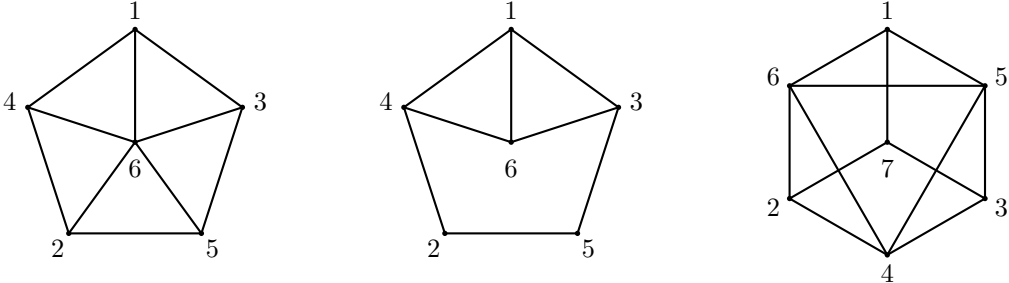

\Cref{main} now follows from \Cref{forbidden_subgraphs} and \Cref{ramsey}.

\begin{proof}[Proof of \Cref{main}]
    Suppose that $L_1,\dots,L_N$ are directed lines in $\mathbb{R}^3$ such that the distance between any two distinct lines is exactly $1$.
    If any two of these lines are parallel, say $L_i$ and $L_j$, then either there is a third line parallel to $L_i$ and $L_j$, in which case there can be no other line equidistant to all 3; or all remaining lines lie in one of the two planes parallel to and at distance 1 from the plane spanned by $L_i$ and $L_j$, in which case there can be at most 4 lines.
    
    Therefore suppose that no two lines are parallel, and consider the 2-coloring of the edges of $K_N$ defined by
    \[
     E(K_n)\to \{\text{red},\text{blue}\};\quad ij\mapsto \begin{cases}
        \text{red}&(\varepsilon(L_i,L_j)=+1)\\
        \text{blue}&(\varepsilon(L_i,L_j)=-1).
    \end{cases}
    \]
    By replacing $L_i$ by $\overline{L_i}$ if necessary, we may assume that all edges adjacent to the vertex $1$ are red.
    If $N\geq 11$, then \Cref{ramsey} shows that there exists an induced monochromatic subgraph isomorphic to one of $K_4$, $K_{3,2}$, $K_6\setminus C_5$, $K_6\setminus H_6$, or $K_7\setminus H_7$ on the vertex set $\{2,3,\dots,n\}$.
    By replacing $L_1$ by $\overline{L_1}$ if necessary, we obtain an induced monochromatic subgraph of $K_N$ isomorphic to one of $K_5$, $K_{3,2}$, $K_7\setminus C_5$, $K_7\setminus H_6$, or $K_8\setminus H_7$, which contradicts \Cref{forbidden_subgraphs}.
    Therefore $N\leq 10$.
\end{proof}

What's left is to verify \Cref{forbidden_subgraphs,ramsey}. We prove \Cref{forbidden_subgraphs} in \Cref{sec:forb-subgraph-proof}. We verify \Cref{ramsey} by computer search; \Cref{pseudocode} outlines our approach. We also provide a fully human-readable proof of \Cref{ramsey} for $K_{12}$ in \Cref{section:12}, which yields a complete proof of $N \leq 12$ that does not rely on computational verification.

\vspace{\baselineskip}
\section{Proof of \texorpdfstring{\Cref{forbidden_subgraphs}}{Theorem 2}}\label{sec:forb-subgraph-proof}


We begin with a slight reformulation of the work in \cite{Koizumi}.

\begin{lemma}\label{lem:realization}
    Let $G$ be a realizable graph with $|V(G)|\geq 2$.
    Then, there exists a real symmetric matrix $T=(a_{vw})_{v,w\in V(G)}$ with the following properties:
    \begin{enumerate}
        \item $a_{vw}=0$ if and only if $v=w$.
        \item $a_{vw}>0$ if and only if $vw\in E(G)$.
        \item $T$ has at most $3$ negative eigenvalues.
        \item The matrix $(|a_{vw}|)_{v,w\in V(G)}$ has signature $(1,n-1)$.
    \end{enumerate}
\end{lemma}

\begin{proof}
    Take a realization $\{L_v\}_{v\in V(G)}$ of $G$.
    Write $L_v=\mathbb{R}x_v+y_v$, where $x_v$ is the unit vector representing the direction of $L_v$.
    Then, we have
    $$
        1=d(L_v,L_w)=\varepsilon(L_v,L_w)\dfrac{\langle x_v\times x_w, y_v-y_w\rangle}{\|x_v\times x_w\|}.
    $$
    Define a matrix $T=(a_{vw})_{v,w\in V(G)}$ by
    $$
        a_{vw}=\varepsilon(L_v,L_w)\|x_v\times x_w\|=\langle x_v\times x_w, y_v-y_w\rangle.
    $$
    Properties (1) and (2) are immediate from the definition. Property (4) is Lemma 3 in \cite{Koizumi}.
    Let us prove (3).
    Define $q_v\in \mathbb{R}^3$ by $q_v=y_v\times x_v$.
    Then, we have
    $$
    a_{vw} = \langle q_v,x_w\rangle + \langle x_v,q_w\rangle=
    \begin{pmatrix}
        q_v^T&x_v^T
    \end{pmatrix}\begin{pmatrix}
        0&I_3\\I_3&0
    \end{pmatrix}\begin{pmatrix}
        q_w\\x_w
    \end{pmatrix}.
    $$
    Therefore, the matrix $T$ is the Gram matrix of the vectors $(q_v,x_v)\in \mathbb{R}^6$ with respect to the symmetric bilinear form of signature $(3,3)$.
    This shows that $T$ has at most $3$ negative eigenvalues.
\end{proof}

\begin{proposition}[{\cite[Theorem 2]{Koizumi}}]\label{prop:K5}
    $K_5$ is not realizable.
\end{proposition}

\begin{proof}
    Assume, for the sake of contradiction, that $G=K_5$ is realizable.
    Take a matrix $T=(a_{vw})_{v,w\in V(G)}$ as in \Cref{lem:realization}.
    Then, $T=(|a_{vw}|)_{v,w\in V(G)}$ has at most $3$ negative eigenvalues but signature $(1,4)$, which is a contradiction.
\end{proof}

The switching $s_wG$ of a graph $G$ at a vertex $w \in V(G)$ refers to the graph obtained by deleting all edges incident to $w$ in $G$ and inserting edges between $w$ and every vertex that was not adjacent to $w$.
Note that if $G$ admits a realization $\{L_v\}_{v\in V(G)}$, then $s_wG$ admits a realization $\{L'_v\}_{v\in V(G)}$, where
$$
L'_v=\begin{cases}
    L_v&(v\neq w)\\
    \overline{L_v}&(v= w).
\end{cases}
$$

\begin{proposition}\label{prop:K32}
    $K_{3,2}$ is not realizable.
\end{proposition}

\begin{proof}
    Fix a labeling of the vertices of $K_{3,2}$ so that the two parts are $\{1,2,3\}$ and $\{4,5\}$.
    Then the graph $\overline{K_5}$, the edgeless graph with $5$ vertices, can be written as $s_4s_5K_{3,2}$.
    Therefore, if $K_{3,2}$ is realizable, then $\overline{K_5}$ must also be realizable, which contradicts \Cref{prop:K5}.
\end{proof}

We now elaborate on these ideas to prove that $K_7 \setminus C_5$, $K_7 \setminus H_6$, and $K_8 \setminus H_7$ are not realizable.

\begin{lemma}\label{lem:cyclic}
    Let $n\geq 3$ be an odd integer.
    Let $A=(a_{ij})_{i,j}$ be an $n\times n$ real symmetric matrix such that
    \begin{align}\label{eq:cyclic}
        \begin{cases}
            a_{ij}>0&(i-j\equiv \pm1\bmod n),\\
            a_{ij}=0&(\text{otherwise}).
        \end{cases}
    \end{align}
    Then, the signature of $A$ is given as follows:
    $$
        \begin{cases}
            \left(\frac{n+1}{2},\frac{n-1}{2}\right)&(n\equiv 1\bmod 4),\\
            \left(\frac{n-1}{2},\frac{n+1}{2}\right)&(n\equiv 3\bmod 4).
        \end{cases}
    $$
\end{lemma}

\begin{proof}
    Let $\mathcal{S}_n$ denote the space of all $n\times n$ symmetric matrices satisfying \eqref{eq:cyclic}.
    For any $A=(a_{ij})_{i,j}\in \mathcal{S}_n$, we have
    $$
        \det A=2\prod_{i=1}^n a_{i,i+1}>0,
    $$
    where $a_{n,n+1}=a_{n,1}$.
    In particular, any member of $\mathcal{S}_n$ is non-singular.
    Since $\mathcal{S}_n$ is connected, the signature is constant on $\mathcal{S}_n$.
    Therefore, it suffices to verify the claim for the adjacency matrix of the $n$-cycle $C_n$.
    In this case, the eigenvalues of $A$ are given by
    $2\cos(2\pi k/n)\ (k=0,1,\dots,n-1)$,
    which proves the claim.
\end{proof}

\begin{proposition}\label{prop:K7-C5}
    $K_7\setminus C_5$ is not realizable.
\end{proposition}

\begin{proof}
    Assume, for the sake of contradiction, that $G:=K_7\setminus C_5$ is realizable.
    Take a matrix $T=(a_{vw})_{v,w\in V(G)}$ as in \Cref{lem:realization}, and define
    $$
        S=(|a_{vw}|)_{v,w\in V(G)},\quad U=S-T.
    $$
    The signature of $S$ is $(1,6)$ by condition (4) of \Cref{lem:realization}.
    In particular, there is a $6$-dimensional subspace $V\subset \mathbb{R}^7$ on which $S$ is negative definite.
    On the other hand, \Cref{lem:cyclic} shows that the signature of $U$ is $(n_+,n_0,n_-)=(3,2,2)$.
    In particular, there is a $5$-dimensional subspace $W\subset \mathbb{R}^7$ on which $U$ is positive semi-definite.
    Since we can write $T=S-U$, we see that $T$ is negative definite on $V\cap W$, whose dimension is at least $4$.
    This contradicts condition (3) of \Cref{lem:realization} which states that $T$ has at most $3$ negative eigenvalues.
\end{proof}

\begin{proposition}\label{prop:K7-H6}
    $K_7\setminus H_6$ is not realizable.
\end{proposition}

\begin{proof}
    Fix a labeling of the vertices of $H_6$ as in \Cref{fig:H6H7}.
    Then, the graph $s_2s_5(K_7\setminus H_6)$ is isomorphic to $K_7\setminus C_5$.
    Therefore, if $K_7\setminus H_6$ is realizable, then $K_7\setminus C_5$ must also be realizable, which contradicts \Cref{prop:K7-C5}.
\end{proof}

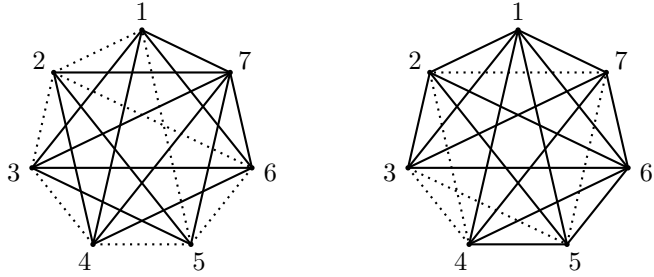
\begin{figure}[h]
    \centering
    \begin{tikzpicture}[scale=0.5]
      \foreach \i [evaluate=\i as \ang using {(\i-1)*360/7 + 90}] in {1,2,3,4,5,6,7}{
        \coordinate (P\i) at (\ang:3);
        \node at (\ang:3.5) {\i};
        \fill (P\i) circle (2pt);
      }
      \foreach \i/\j in {1/3,1/4,1/6,1/7,2/4,2/5,2/7,3/5,3/6,3/7,4/6,4/7,5/7,6/7}{
        \draw[thick] (P\i) -- (P\j);
      }
      \foreach \i/\j in {1/2,2/3,3/4,4/5,5/1,2/6,5/6}{
        \draw[thick, dotted] (P\i) -- (P\j);
      }
      \begin{scope}[xshift=10cm]
      \foreach \i [evaluate=\i as \ang using {(\i-1)*360/7 + 90}] in {1,2,3,4,5,6,7}{
        \coordinate (P\i) at (\ang:3);
        \node at (\ang:3.5) {\i};
        \fill (P\i) circle (2pt);
      }
      \foreach \i/\j in {1/2,1/3,1/4,1/5,1/6,1/7,2/3,2/5,2/6,3/6,3/7,4/5,4/6,4/7,5/6,6/7}{
        \draw[thick] (P\i) -- (P\j);
      }
      \foreach \i/\j in {2/4,4/3,3/5,5/7,7/2}{
        \draw[thick, dotted] (P\i) -- (P\j);
      }
      \end{scope}
    \end{tikzpicture}
    \caption{$K_7\setminus H_6$ (left) and $s_2s_5(K_7\setminus H_6)$ (right)}
    \label{fig:K7H6_switching}
\end{figure}

\begin{lemma}\label{lem:h7}
    Let $A=(a_{ij})_{i,j}$ be a $7\times 7$ real symmetric matrix such that
    \begin{align}\label{eq:h7}
        \begin{cases}
            a_{ij}>0&(ij\in E(H_7)),\\
            a_{ij}=0&(\text{otherwise}).
        \end{cases}
    \end{align}
    Then, $A$ has signature $(4,3)$.
\end{lemma}

\begin{proof}
    Let $\mathcal{S}$ denote the space of all $7\times 7$ symmetric matrices satisfying \eqref{eq:h7}.
    Let
    $$
    A=\begin{pmatrix}
        &a_{12}&a_{13}&a_{14}&&&\\
        a_{12}&&a_{23}&&a_{25}&&\\
        a_{13}&a_{23}&&&&a_{36}&\\
        a_{14}&&&&&&a_{47}\\
        &a_{25}&&&&&a_{57}\\
        &&a_{36}&&&&a_{67}\\
        &&&a_{47}&a_{57}&a_{67}&
    \end{pmatrix}\in \mathcal{S}.
    $$
    Then we have
    $$
    \det A
    = -2 a_{14}a_{25}a_{36} ( a_{12}a_{36}a_{47}a_{57}
    + a_{13}a_{25}a_{47}a_{67}
    + a_{14}a_{23}a_{57}a_{67})<0.
    $$
    In particular, any member of $\mathcal{S}$ is non-singular.
    Since $\mathcal{S}$ is connected, the signature is constant on $\mathcal{S}$.
    Therefore, it suffices to verify the claim for the adjacency matrix of $H_7$.
    In this case, a direct computation shows that $A$ has signature $(4,3)$.
\end{proof}

\begin{proposition}\label{prop:K8-H7}
    $K_8\setminus H_7$ is not realizable.
\end{proposition}

\begin{proof}
    Assume, for the sake of contradiction, that $G:=K_8\setminus H_7$ is realizable.
    Take a matrix $T=(a_{vw})_{v,w\in V(G)}$ as in \Cref{lem:realization}, and define
    $$
        S=(|a_{vw}|)_{v,w\in V(G)},\quad U=S-T.
    $$
    The signature of $S$ is $(1,7)$ by condition (4) of \Cref{lem:realization}.
    In particular, there is a $7$-dimensional subspace $V\subset \mathbb{R}^8$ on which $S$ is negative definite.
    On the other hand, \Cref{lem:h7} shows that the signature of $U$ is $(n_+,n_0,n_-)=(4,1,3)$.
    In particular, there is a $5$-dimensional subspace $W\subset \mathbb{R}^8$ on which $U$ is positive semi-definite.
    Since we can write $T=S-U$, we see that $T$ is negative definite on $V\cap W$, whose dimension is at least $4$.
    This contradicts condition (3) of \Cref{lem:realization} which states that $T$ has at most $3$ negative eigenvalues.
\end{proof}

\section{Concluding remarks}

In this paper, we showed by a Ramsey-theoretic approach that at most 10 cylinders can lie in a mutually touching configuration. Combined with the lower bound of Boz\'oki, Lee, and R\'onyai \cite{BLR}, it follows that the maximum size of a mutually touching set of cylinders is $7,8,9,$ or $10$. We suspect that the right answer is 7, since that is the largest value for which the number of degrees of freedom in the line configurations is larger than the number of imposed constraints. 

The same problem is wide-open in higher dimensions. Let $\Ch(n)$ denote the \emph{champagne glass number in $\mathbb{R}^n$}, that is, the maximum size of a collection of lines in $\mathbb{R}^n$ in which the distance between each pair of lines is exactly 1. The following construction yields the lower bound $\Ch(n)\geq 2n-2$. Let $v_1,\dots,v_{n-1}\in  \mathbb{R}^{n-2}$ be the vertices of an $(n-2)$-simplex of side length $1$. For each $i=1,2,\dots,n-1$, choose two parallel lines $L_i^{(1)},L_i^{(2)}$ in $\mathbb{R}^2$ at distance $1$, arranged so that for distinct $i$ the chosen pairs are not parallel to each other. Then
$$
\{v_i\}\times L_i^{(j)}\quad (i=1,2,\dots,n-1,\ j=1,2)
$$
gives a set of $2n-2$ equidistant lines. In high dimensions, this is probably a weak bound; in fact, it is not even clear whether $\Ch(n)$ is finite for all $n$.

\section*{Acknowledgements}
TD's work was partially supported by a National Science Foundation Graduate Research Fellowship under Grant No.\@ 2141064. TD also thanks Jessica Liu, who, after hearing about geometric equidistance problems, piqued his interest by asking ``What about lines?''

SL's research is supported by NSF Award No.\@ 2303290. SL would also like to thank his friends Daniel and Lily, at whose wedding banquet he first came across this problem thanks to their tall (though not quite infinite) cylindrical champagne glasses. May their decades of happiness far exceed $\Ch(3)$.

\printbibliography
\addresseshere

\newpage


\appendix

\section{Algorithm for \texorpdfstring{\Cref{ramsey}}{Theorem 3}}\label{pseudocode}

In principle, one could prove \Cref{ramsey} by exhaustive search, finding an induced monochromatic copy of $K_4, K_{3,2}, K_6 \setminus C_5, K_6 \setminus H_6$, or $K_7 \setminus H_7$ in each 2-coloring of $K_{10}$. Rather than search through all colorings, we speed things up by iteratively building up ``feasible subgraphs''. 

We let $\mathcal F_k$ denote the set of 2-colorings of $K_k$ that do not contain a forbidden subgraph. To get $\mathcal F_{k+1}$ from $\mathcal F_k$, we use the following algorithm:
\begin{enumerate}
    \item For each coloring in $\mathcal F_k$, make $2^k$ new colorings by adding a $(k+1)$th vertex connected to all $k$ previous vertices, with all possible 2-colorings of the new edges. Let $\mathcal G_{k+1}$ denote the set of all these new colorings created from the colorings in $\mathcal F_k$.
    \item $\mathcal F_{k+1}$ consists of all those colorings in $\mathcal G_{k+1}$ that do not contain a forbidden subgraph.
\end{enumerate}

This process works because any coloring on $K_{k+1}$ with no forbidden subgraph can be constructed by adding a vertex to a coloring of $K_k$ vertices with no forbidden subgraph.

In our program, the two colors of edges are represented by edges and non-edges. The statement that every 2-coloring of $K_{10}$ contains a forbidden graph in one of the colors is equivalent to the statement that every graph on 10 vertices contains a forbidden graph or its complement as an induced subgraph. Our program proceeds via the following functions:

\vspace{0.5\baselineskip}
\begin{description}[font=\ttfamily]
    \item[IsForbidden($G$)] returns \texttt{True} if the graph $G$ contains a forbidden subgraph and \texttt{False} otherwise.
    \item[AllowedGraphs($L$)] takes a list of graphs $L$ and returns a new list that contains all graphs in $L$ that do not contain a forbidden subgraph.
    \item[AddVertex($G$)] takes a graph $G$ and returns a list of graphs constructed by adding a new vertex and all possible 2-colorings of the new edges.
    \item[AddVertexList($L$)] takes a list of graphs $L$ and returns the list obtained by applying \texttt{AddVertex} to each graph in $L$.
    \item[OneMoreVertex($L$)] takes a list of graph $L$, applies \texttt{AddVertexList} to $L$, and removes isomorphic duplicates.
\end{description}
\vspace{0.5\baselineskip}

The program runs by initializing with $L = \{K_1\}$ and applying \texttt{OneMoreVertex} 9 times. After approximately 3 hours (on a standard laptop), the program returns the empty list, signifying that every 2-coloring of $K_{10}$ contains a forbidden graph.

The Mathematica code we used to implement this procedure is available online \cite{code}.

\vspace{\baselineskip}
\section{Human-readable proof of \texorpdfstring{$\Ch(3)\le 12$}{Ch(3) <= 12}}\label{section:12}

We provide a fully human-readable proof of the following weaker version of \Cref{ramsey}.
This yields a complete proof of $\Ch(3) \leq 12$ that does not rely on computational verification.

\begin{theorem}
    In every 2-coloring of the edges of $K_{12}$, there exists a monochromatic induced subgraph isomorphic to one of $K_4$, $K_{3,2}$, $K_6\setminus C_5$, $K_6\setminus H_6$, or $K_7\setminus H_7$.
\end{theorem}

\begin{proof}
    Fix a 2-coloring of $E(K_{12})$ in red and blue. Assume for the sake of contradiction that this coloring contains no monochromatic induced subgraph isomorphic to one of $K_4$, $K_{3,2}$, $K_6\setminus C_5$, $K_6\setminus H_6$, or $K_7\setminus H_7$. Fix an arbitrary vertex $o$, and consider its red and blue neighborhoods $A=N_r(o)$ and $B=N_b(o)$. Without loss of generality, let $|A|\le |B|$, so $|B|\ge 6$. We will break into cases based on the size of the blue neighborhood $B$.

\vspace{5pt}

\noindent \textbf{Case 1: $|B|\ge 9$.} Recall that $R(3,4)=9$, so we obtain a monochromatic $K_4$ in $B\cup \{o\}$ in this case.

\vspace{5pt}

\noindent \textbf{Case 2: $|B|=8$.} If $B$ contains a blue $C_3$ or a blue $C_5$, then $B\cup \{o\}$ contains a blue $K_4$ or an induced blue $K_6\setminus C_5$, and we are done. Otherwise, either $B$ contains a blue $C_7$, or the blue graph on $B$ is bipartite. In the latter case, one of the two parts in the bipartition has size at least $4$, yielding a red $K_4$. 

Thus, we can assume that $B$ contains a blue $C_7$. We can further assume this $C_7$ is induced (otherwise $B$ would also contain a $C_3$ or $C_5$ in blue). Let the vertices of this blue cycle be $v_1,\dots,v_7$ in order. For convenience, we let $v_{i+7}=v_i$ for all $i$. Let $v^*$ be the eighth vertex of $B$. If $v^*$ is adjacent in blue to a pair of consecutive vertices $v_i, v_{i+1}$ of the blue $C_7$, then $\{o,v_i,v_{i+1},v^*\}$ induces a blue $K_4$. Otherwise, it is easy to check that $v^*$ is adjacent in red to three pairwise nonconsecutive vertices of the induced blue $C_7$, meaning we have a red $K_4$ and are again done.

\vspace{5pt}

\noindent \textbf{Case 3: $|B|=7$.} As in the previous case, we can reduce to the case where there is an induced blue $C_7$ in $B$ with vertices $v_1,\dots,v_7$ in order, and let $v_{i+7}=v_i$ for all $i$ for convenience.

\begin{figure}[h]
    \centering
    \begin{tikzpicture}[scale=0.5]
      \foreach \i [evaluate=\i as \ang using {(\i-1)*360/7 + 90}] in {1,2,3,4,5,6,7}{
        \coordinate (v\i) at (\ang:3);
        \node at (\ang:3.5) {$v_\i$};
        \fill (v\i) circle (2pt);
      }
      \coordinate (v8) at (-8,-2);
      \node at (-8,-2) [below] {$w$};
      \fill (v8) circle (2pt);
      \foreach \i/\j in {1/2, 2/3, 3/4, 4/5, 5/6, 6/7, 7/1}{
        \draw[thick,blue] (v\i) -- (v\j);
      }
	  \foreach \i/\j in {8/7, 8/1, 8/4, 4/7, 4/1, 4/2, 4/6, 2/7, 1/6, 2/6}{
        \draw[thick,red] (v\i) -- (v\j);
      }
    \end{tikzpicture}
    \caption{Case 3, red $K_6\setminus C_5$ centered at $v_4$
    }
    \label{fig:Case3a}
\end{figure}
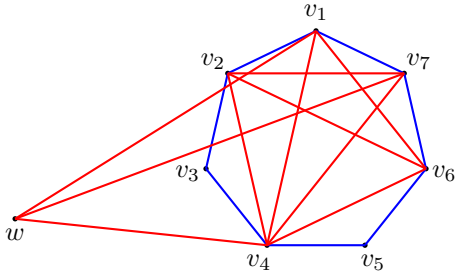

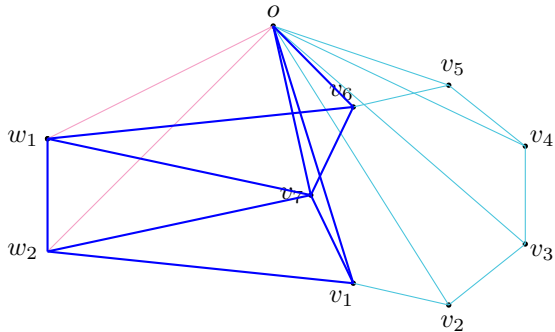
\begin{figure}[h]
    \centering
    \begin{tikzpicture}[scale=0.5]
      \foreach \i [evaluate=\i as \ang using {(\i)*360/7 + 180}] in {1,2,3,4,5,6,7}{
        \coordinate (v\i) at (\ang:3);
        \node at (\ang:3.5) {$v_\i$};
        \fill (v\i) circle (2pt);
      }
      \coordinate (w1) at (-10,1.5);
      \node at (-10,1.5) [left] {$w_1$};
      \fill (w1) circle (2pt);
      \coordinate (w2) at (-10,-1.5);
      \node at (-10,-1.5) [left] {$w_2$};
      \fill (w1) circle (2pt);
      \coordinate (v8) at (-4,4.5);
      \node at (-4,4.5) [above] {$o$};
      \fill (v8) circle (2pt);
	   \foreach \i/\j in {8/1,8/2}{
        \draw[Lavender] (v\i) -- (w\j);
      }
      \foreach \i/\j in {1/2, 2/3, 3/4, 4/5, 5/6, 8/2, 8/3, 8/5, 8/4}{
        \draw[SkyBlue] (v\i) -- (v\j);
      }
      \foreach \i/\j in {8/1,8/7,8/6, 6/7, 7/1}{
        \draw[thick,blue] (v\i) -- (v\j);
      }
      \draw[thick,blue] (w1) -- (w2);
	  \foreach \i/\j in {6/1,7/1,7/2,1/2}{
        \draw[thick,blue] (v\i) -- (w\j);
      }
    \end{tikzpicture}
    \caption{Case 3, blue $K_6\setminus C_5$ on $\{o,w_1,w_2,v_6,v_7,v_1\}$ centered at $v_7$
    }
    \label{fig:Case3b}
\end{figure}

\begin{figure}[ht]
    \centering
    \begin{tikzpicture}[scale=0.5]
      \foreach \i [evaluate=\i as \ang using {(\i+0.5)*360/7 + 180}] in {1,2,3,4,5,6,7}{
        \coordinate (v\i) at (\ang:3);
        \node at (\ang:3.5) {$v_\i$};
        \fill (v\i) circle (2pt);
      }
      \coordinate (w1) at (-10,1.5);
      \node at (-10,1.5) [left] {$w_1$};
      \fill (w1) circle (2pt);
      \coordinate (w2) at (-10,-1.5);
      \node at (-10,-1.5) [left] {$w_2$};
      \fill (w1) circle (2pt);
      \coordinate (v8) at (-3.5,4.5);
      \node at (-3.5,4.5) [above] {$o$};
      \fill (v8) circle (2pt);
      \foreach \i/\j in {1/2, 2/3, 3/4, 4/5, 5/6, 6/7, 7/1, 8/2, 8/3, 8/5, 8/7}{
        \draw[SkyBlue] (v\i) -- (v\j);
      }
      \foreach \i/\j in {5/1,7/2,2/2}{
        \draw[Lavender] (v\i) -- (w\j);
      }
    	  \foreach \i/\j in {5/2,7/1,2/1}{
        \draw[SkyBlue] (v\i) -- (w\j);
      }
      \foreach \i/\j in {8/1,8/4,8/6}{
        \draw[thick,blue] (v\i) -- (v\j);
      }
      \foreach \i/\j in {6/1,1/4,4/6}{
        \draw[thick,red] (v\i) -- (v\j);
      }
      \draw[thick,blue] (w1) -- (w2);
	  \foreach \i/\j in {6/1,4/1,1/2,8/1,8/2}{
        \draw[thick,red] (v\i) -- (w\j);
      }
	  \foreach \i/\j in {6/2,4/2,1/1}{
        \draw[thick,blue] (v\i) -- (w\j);
      }
    \end{tikzpicture}
    \caption{Case 3, induced blue $K_6\setminus H_6$ on $\{o,w_1,w_2,v_1,v_4,v_6\}$
    }
    \label{fig:Case3c}
\end{figure}
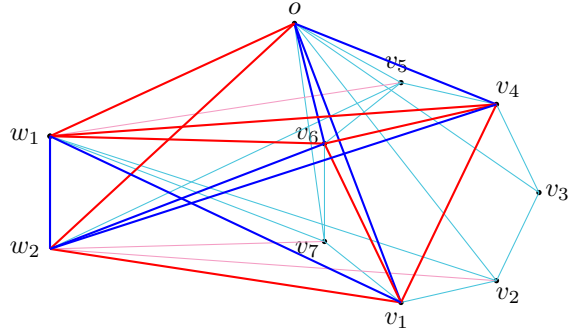

Fix a vertex $w\in A$, and consider the colors of its edges to the vertices of $B$. For every triple of pairwise nonconsecutive vertices of the blue $C_7$ in $B$, say $v_i,v_{i+2},v_{i+4}$, if $w$ is connected to all three in red, then we obtain a red $K_4$; if instead $w$ is connected to all three in blue, then $\{o,w,v_i,v_{i+2},v_{i+4}\}$ forms an induced blue $K_{3,2}$. 
Similarly, if $w$ is adjacent to $v_i,v_{i+1},v_{i+4}$ in red, then $\{w,v_{i+4},v_i,v_{i+1},v_{i+2},v_{i+6}\}$ forms a red copy of $K_6\setminus C_5$ centered at $v_{i+4}$ (see \Cref{fig:Case3a}), which either is induced or contains a red $K_4$. We can thus assume that none of these scenarios occur, from which we can conclude that the red neighbors of $w$ in $B$ form a consecutive block along the blue $C_7$ and that they number either $3$ or $4$ (and equivalently, the same is true of the blue neighbors of $w$ in $B$).

Since $|A|\ge 3$, there must be two vertices $w_1,w_2\in A$ such that the edge between them is blue (otherwise there would be a red $K_4$ among $A\cup \{o\}$). If $w_1,w_2$ have a common blue neighbor in $B$, then since their blue neighborhoods in $B$ each forms a block of consecutive vertices along the blue $C_7$, there is some $i$ such that (without loss of generality) $w_1$ is connected in blue to $v_i$ and $v_{i-1}$, while $w_2$ is connected in blue to $v_i$ and $v_{i+1}$. Then $\{w_1,w_2,o,v_{i-1},v_i,v_{i+1}\}$ forms a blue copy of $K_6\setminus C_5$ centered at $v_i$ (see \Cref{fig:Case3b}). Thus the blue neighborhoods of $w_1,w_2$ in $B$ must be disjoint.

We can then find some $i$ such that (without loss of generality) $v_i,v_{i+1},v_{i+2}$ are in the blue neighborhood of $w_1$ and the red neighborhood of $w_2$, while $v_{i-1},v_{i-2},v_{i-3}$ are in the blue neighborhood of $w_2$ and the red neighborhood of $w_1$. Then $\{o,w_1,w_2,v_{i+1},v_{i-1},v_{i-3}\}$ forms an induced blue copy of $K_6\setminus H_6$ (see \Cref{fig:Case3c}).

\vspace{5pt}

\noindent \textbf{Case 4: $|B|=6$.} In this case, since there again cannot be any blue $C_3$ or $C_5$ or red $K_4$ in $B$, the blue graph on $B$ must be bipartite with exactly $3$ vertices in each part. Likewise, the red graph on $A$ is bipartite with $3$ vertices in one part and $2$ in the other. Fix a red triangle $T=\{v_1,v_2,v_3\}\subset B$ and a blue triangle $S=\{w_1,w_2,w_3\}\subset A$.  

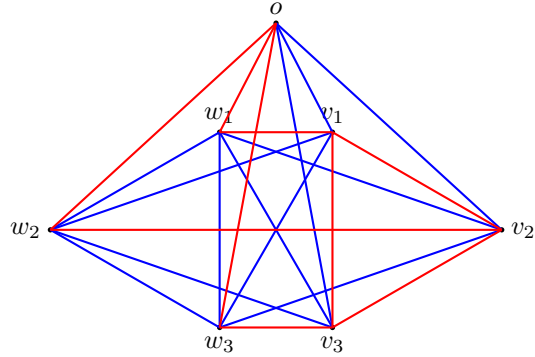
\begin{figure}[h]
    \centering
    \begin{tikzpicture}[scale=0.5]
      \coordinate (v1) at (60:3);
      \node at (60:3) [above] {$v_1$};
      \fill (v1) circle (2pt);
      \coordinate (v2) at (0:6);
      \node at (0:6) [right] {$v_2$};
      \fill (v2) circle (2pt);
      \coordinate (v3) at (-60:3);
      \node at (-60:3) [below] {$v_3$};
      \fill (v3) circle (2pt);
      \coordinate (v4) at (120:3);
      \node at (120:3) [above] {$w_1$};
      \fill (v4) circle (2pt);
      \coordinate (v5) at (180:6);
      \node at (180:6) [left] {$w_2$};
      \fill (v5) circle (2pt);
      \coordinate (v6) at (-120:3);
      \node at (-120:3) [below] {$w_3$};
      \fill (v6) circle (2pt);

      \coordinate (v7) at (0,5.5);
      \node at (0,5.5) [above] {$o$};
      \fill (v7) circle (2pt);
      \foreach \i/\j in {7/1,7/2,7/3,4/5,5/6,6/4,1/5,2/4,3/4,3/5,1/6,2/6}{
        \draw[thick,blue] (v\i) -- (v\j);
      }
      \foreach \i/\j in {7/4,7/5,7/6,1/2,2/3,3/1,1/4,2/5,3/6}{
        \draw[thick,red] (v\i) -- (v\j);
      }
    \end{tikzpicture}
    \caption{Case 4, induced blue $K_7\setminus H_7$
    }
    \label{fig:Case4}
\end{figure}

If any $w\in S$ is connected in red to all of $T$, then we obtain a red $K_4$; if $w$ is connected in blue to all of $T$, then $\{o,w\}\cup T$ forms an induced blue $K_{3,2}$. Thus every vertex of $S$ must have both red and blue edges to $T$, and vice versa. The blue neighborhoods in $T$ of the vertices of $S$ then cannot be ordered by inclusion (else some vertex is in all of them, i.e. some vertex of $T$ is connected in blue to all of $S$), which means there are two vertices in $S$, say $w_1,w_2$, and two vertices in $T$, say $v_1,v_2$, such that $w_1,v_1$ and $w_2,v_2$ are adjacent in red, while $w_1,v_2$ and $w_2,v_1$ are adjacent in blue. Note that $\{o,w_1,v_1,v_2,w_2\}$ forms an induced red $C_5$. In order for $\{o,w_1,v_1,v_2,w_2,v_3\}$ not to form an induced red $K_6\setminus H_6$, the edges from $v_3$ to $w_1,w_2$ must be both red or both blue, and likewise for the edges from $w_3$ to $v_1,v_2$. Neither pair of edges can agree in color with the edge between $v_3$ and $w_3$, which by symmetry we can assume without loss of generality to be red. Thus the coloring on $\{o\}\cup S\cup T$ is as depicted in \Cref{fig:Case4}. But this yields an induced blue $K_7\setminus H_7$, a contradiction. Thus in every case, we cannot avoid all of our forbidden induced subgraphs in a $2$-coloring of the edges of $K_{12}$.
\end{proof}

\end{document}